\documentclass{amsart}
\usepackage{amscd,amsfonts,amssymb,amsmath}
\usepackage[margin=3.9cm]{geometry}

\usepackage{graphics}
\usepackage{epsfig}
\usepackage{amsmath}

\newtheorem{theorem}{Theorem}[section]

\theoremstyle{definition}
\newtheorem{definition}[theorem]{Definition}

\newtheorem{prop}[theorem]{Proposition}
\theoremstyle{remark}
\newtheorem{remark}[theorem]{Remark}
\numberwithin{equation}{subsection}
\theoremstyle{plain}

\newtheorem{question}{Question}

\numberwithin{equation}{section}

\begin{document}

\title[MONOID AND GROUP OF PSEUDO BRAIDS]{MONOID AND GROUP OF PSEUDO BRAIDS}

\author{Valeriy ~G.~Bardakov}
\address{Sobolev Institute of Mathematics, Novosibirsk State University, Novosibirsk 630090, Russia
and Laboratory of Quantum Topology, Chelyabinsk State University, Brat'ev Kashirinykh street 129, Chelyabinsk 454001, Russia}
\email{bardakov@math.nsc.ru}
\author{Slavik Jablan}
\address{The Mathematical Institute, Belgrade, 11000, Serbia}
\email{sjablan@gmail.com}
\author{Hang Wang}
\address{School of Mathematical Sciences, University of Adelaide, Adelaide 5005, Australia}
\email{hang.wang01@adelaide.edu.au}

\subjclass[2000]{Primary 20F18; Secondary 20D15,20E05}
\keywords{the theory of pseudo knots, pseudo braids, singular knots, Markov's Theorem}

\thanks{\bf The first author  is partially supported by Laboratory of Quantum Topology of Chelyabinsk State University (Russian Federation government grant 14.Z50.31.0020)}

\date{\today}


\begin{abstract}
In the present paper we define a monoid of pseudo braids and prove that this monoid is isomorphic to a singular
braid monoid.  Also we prove an analogue of Markov's Theorem for pseudo braids.
\end{abstract}
\maketitle

\section{Introduction}

In 2010 Ryo Hanaki \cite{H} introduced the notion of a pseudo diagram of a knot, link and spatial graph. A pseudo knot diagram is similar to the projection of a knot in $\mathbb{R}^2$, but besides over-crossings and under-crossings we allow unspecified crossings (which could be either over or under crossings). A double point with over/under information and a double point without over/under information are called a {\it crossing} and a {\it pre-crossing}, respectively. The notion of a pseudo diagram is a generalization of a knot or link diagram. The observation of DNA knots was an opportunity of this research, namely we can not determine over/under information at some of the crossings in some photos of DNA knots. DNA knots barely become visual objects by examining the protein-coated one by electromicroscope. However there are still cases in which it is hard to confirm the over/under information of some of the crossings. If we know the (non-)triviality of knot without checking every over/under information of crossings, then it may give a reasonable way to detect the (non-)triviality of DNA knot.

The notion of a pseudo knot was subsequently introduced in \cite{J} and defined as an equivalence class of pseudo diagrams under an appropriate choice of Reidemeister moves. In order to classify pseudo knots the
authors of \cite{J} introduced the concept of a weighted resolution set, an invariant of pseudo knots, and computed the
weighted resolution set for several pseudo knots families and discussed extensions of crossing number, homotopy,
and chirality for pseudo knots. The aim of this note is to address the following question formulated in their paper.

\begin{question} (\cite[Question 3]{J}) What is an appropriate definition of pseudo braids? In particular, when
are two pseudo braids equivalent? Furthermore, in classical braid theory there are Markov moves that characterize
when two braids have equivalent closure. Is there an analog for pseudo braids?
\end{question}

Pseudo knots are closely related to singular knots.
There is a map $f$ from the set of singular knot diagrams to the set of pseudo knot diagrams by replacing singular crossings by pseudo crossings. The map $f$ is bijective because for every pseudo knot diagram, we can find a singular knot diagram where the singular crossings replace the pseudo crossings and vice versa. Modulo equivalence relation defining pseudo or singular knots, the map $f$ induces an onto map, denoted $F$, from singular knots/links to pseudo knots/links because the image of two isotopic singular knot diagrams are also isotopic pseudo knot diagrams with exactly the same sequence of Reidemeister moves (where the singular crossings are replaced by pseudo crossings):
\begin{equation}
\label{eq:F}
F: S\mathcal{L}\rightarrow P\mathcal{L}.
\end{equation}
Respectively, $S\mathcal{L}, P\mathcal{L}$ denote the set of singular links and that of pseudo links.
But note that this map $F$ is not one to one. The key idea for this note is to present a pseudo knot $K$ using a singular knot in the pre image $F^{-1}(K)$ and using the existing formulation of singular knots to deduce our result for pseudo knots.

\section{Monoid of Pseudo Braids}

To properly define the monoid of pseudo braids, we recall the definition of a singular monoid.

\begin{definition}[Baez-Birman \cite{Baez, B1}]
The Baez-Birman monoid or the singular braid monoid $SM_n$ is generated (as a monoid) by elements
$\sigma_i$, $\sigma_i^{-1}$, $\tau_i$, $i = 1, 2, \ldots, n-1$ subject to the following relations:
the relations for the elements $\sigma_i$, $\sigma_i^{-1}$
generating the braid group $B_n$; defining relations for the generators $\tau_i$:
 \begin{equation}\label{11}
\tau_i \tau_j = \tau_j \tau_i,~~~|i - j| \geq 2,
\end{equation}
and the other mixed relations:
 \begin{equation}\label{22}
\tau_i \sigma_j^{\pm 1} = \sigma_j^{\pm 1} \tau_i,~~~|i - j| \geq 2,
\end{equation}
 \begin{equation}\label{33}
\tau_i \sigma_i^{\pm 1} = \sigma_i^{\pm 1} \tau_i,~~~i = 1, 2, \ldots, n-1,
\end{equation}
 \begin{equation}\label{44}
\sigma_i \sigma_{i+1} \tau_i = \tau_{i+1} \sigma_i \sigma_{i+1},
\end{equation}
 \begin{equation}\label{55}
\sigma_{i+1} \sigma_{i} \tau_{i+1} = \tau_{i} \sigma_{i+1} \sigma_{i}.
\end{equation}
\end{definition}
Fenn-Keyman-Rourke \cite{FKR} proved that the singular braid monoid $SM_n$ is embedded into the group $SB_n$ which
is called {\it the singular braid group} and has the same defining relations as $SM_n$.

The generators $\tau_i$, $i = 1, 2, \ldots, n-1$ have a geometric interpretation (see Figure 1).
\begin{center}
\includegraphics[width=8.0cm]{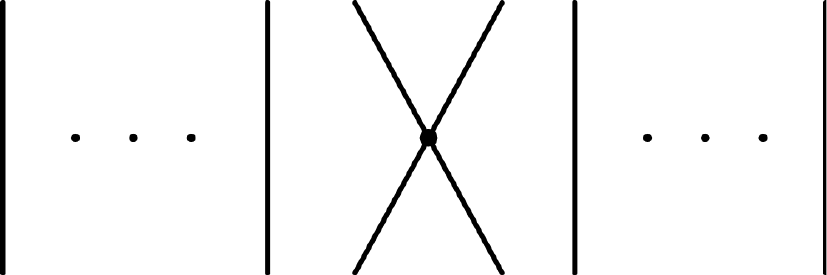}\\
{Figure 1: The generator $\tau_i$}
\end{center}
where the $i$-th and $i+1$-th strings intersect.

To define a monoid of pseudo braid $PM_n$ we take the generators $\sigma^{\pm 1}_1, \sigma_2^{\pm 1}, \ldots, \sigma^{\pm 1}_{n-1}$ of the braid group $B_n$ and add the generators $p_1, p_2, \ldots, p_{n-1}$ which are similar to the generators $\tau_1, \tau_2, \ldots, \tau_{n-1}$ and they denote pre-crossings instead of singular crossings. To find defining relations we make use of pseudo-Reidemeister moves for the pseudo links (see Figure 2 in \cite{J}). The Reidemeister moves R2 and R3 correspond to relations
 \begin{equation}\label{222}
\sigma_i \sigma_i^{-1} = \sigma_i^{-1} \sigma_i,
\end{equation}
 \begin{equation}\label{555}
\sigma_{i+1} \sigma_{i} \sigma_{i+1} = \sigma_{i} \sigma_{i+1} \sigma_{i}.
\end{equation}
Note that these are relations in $B_n$.

The move PR2 (see Figure 2) corresponds to relations
 \begin{equation}\label{33}
\tau_i \sigma_i^{\pm 1} = \sigma_i^{\pm 1} \tau_i.
\end{equation}

The moves PR3 and RP3' (see Figure 3 and Figure 4) correspond to relations
 \begin{equation}\label{44}
\sigma_i \sigma_{i+1} \tau_i = \tau_{i+1} \sigma_i \sigma_{i+1},
\end{equation}
 \begin{equation}\label{55}
\sigma_{i+1} \sigma_{i} \tau_{i+1} = \tau_{i} \sigma_{i+1} \sigma_{i}.
\end{equation}

The first Reidemeister moves R1 and PR1 do not hold in $PM_n$. Hence, we can formulate:

\begin{definition}[Monoid and group of pseudo braids]
The {\it monoid of pseudo braids} $PM_n$ is a monoid generated by
$\sigma_i$, $\sigma_i^{-1}$, $p_i$, $i = 1, 2, \ldots, n-1$, where the elements $\sigma_i^{\pm 1}$ generate
the braid group $B_n$ and generators $p_i$ satisfy the defining relations~
\begin{equation}\label{1}
p_i p_j = p_j p_i,~~~|i - j| \geq 2,
\end{equation}
 and other mixed relations
  \begin{equation}\label{2}
p_i \sigma_j^{\pm 1} = \sigma_j^{\pm 1} p_i,~~~|i - j| \geq 2,
\end{equation}
 \begin{equation}\label{3}
p_i \sigma_i^{\pm 1} = \sigma_i^{\pm 1} p_i,~~~i = 1, 2, \ldots, n-1,
\end{equation}
 \begin{equation}\label{4}
\sigma_i \sigma_{i+1} p_i = p_{i+1} \sigma_i \sigma_{i+1},
\end{equation}
 \begin{equation}\label{5}
\sigma_{i+1} \sigma_{i} p_{i+1} = p_{i} \sigma_{i+1} \sigma_{i}.
\end{equation}

The {\it group of pseudo braids} $PB_n$ is a group  generated by
$\sigma_i$,  $p_i$, $i = 1, 2, \ldots, n-1$ and defined by the same defining relations as $PM_n$.
\end{definition}

\begin{center}
\includegraphics[width=5.0cm]{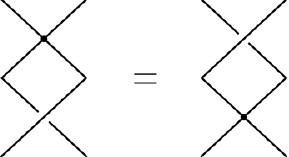}\\
{Figure 2: The Reidemeister move RP2}
\end{center}
\begin{center}
\includegraphics[width=10.0cm]{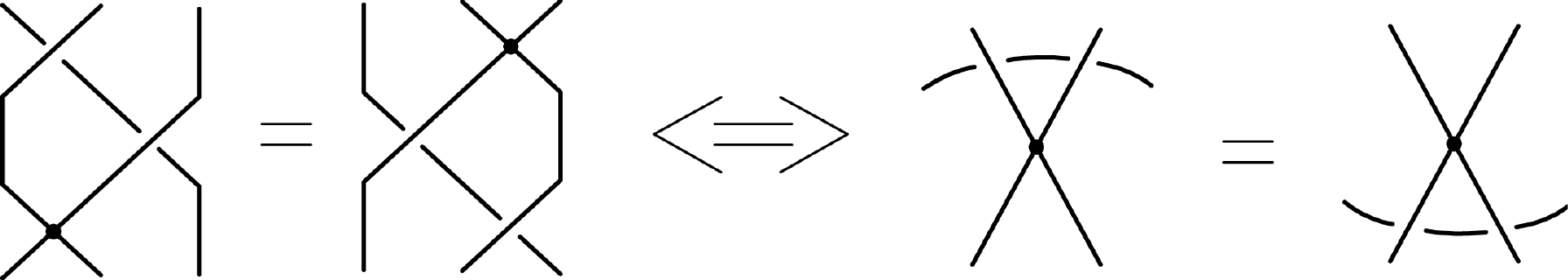}\\
{Figure 3: The Reidemeister move RP3}
\end{center}
\begin{center}
\includegraphics[width=10.0cm]{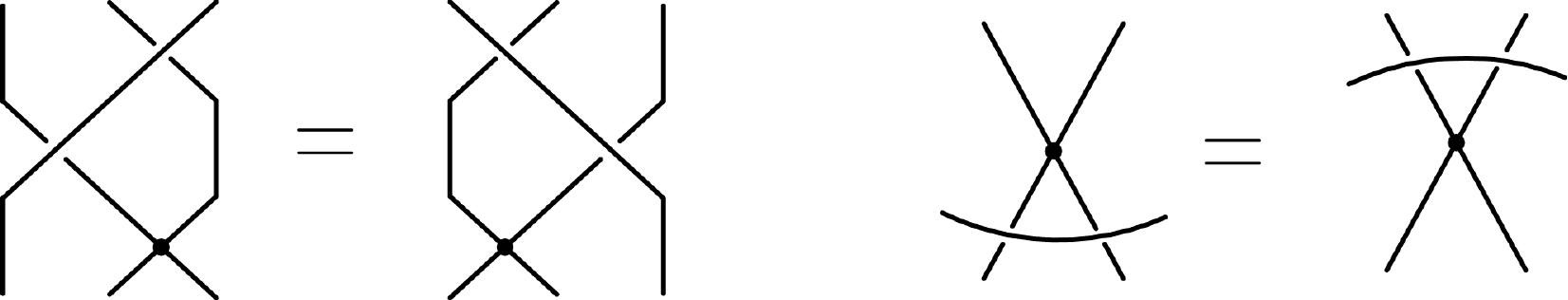}\\
{Figure 4: The Reidemeister move RP3'}
\end{center}

Comparing the defining relations for $SM_n$ and $PM_n$ we see that they are isomorphic and the isomorphism $SM_n \longrightarrow PM_n$ is defined by the rule $\sigma_i \mapsto \sigma_i$ and $\tau_i \mapsto p_i$ for all $i = 1, 2, \ldots, n-1$. On the other side,
in \cite{FKR} it was proved that  $SM_n$ is embedded into the the group $SB_n$ which is called the {\it singular braid group} and has the same defining relations as $SM_n$. Hence, we have:

\begin{prop}\label{con}
The monoid of pseudo braids $PM_n$ is isomorphic to the singular braid monoid $SM_n$ and the group of pseudo braids $PB_n$ is isomorphic to the group of singular braids $SB_n$ for all $n \geq 2$.
\end{prop}

\begin{remark}
Comparing the Reidemeister moves for the singular links and for pseudo links, we see that the difference is the first Reidemeister move PR1 (see Figure 2 in \cite{J}) which does not hold for singular links but holds for pseudo links. Hence, the Theory of pseudo links is the quotient of the Theory of singular links by the first singular Reidemeister move.
\end{remark}

Since the monoid of pseudo braids $PM_n$ is isomorphic to the singular braid monoid $SM_n$, we can reformulate properties of $SM_n$ for $PM_n$.  There are some different ways to solve the word problem for the singular braid monoid and for the singular braid group. The singular braid monoid on two strings is isomorphic to $\mathbb{Z} \times \mathbb{Z}^+$, so the word problem in this case is trivial. In the general case a solution of the word problem
in $SM_n$ was done by R. Corran \cite{Cor}. V. Vershinin \cite{V} generalized Garside's results and constructed the greedy normal form for $SM_n$. We don't know if it is possible to generalize this result to the group $SB_n$.
L.~Paris \cite{P}
proved that $SM_n$ is a semi-direct product of a right-angled Artin monoid and $B_n$. This gives a solution to the word problem for $SM_n$ and $SB_n$.

The desingularization map is the multiplicative homomorphism $\eta : SM_n \longrightarrow \mathbb{Z}[B_n]$ defined by $\eta(\sigma_i^{\pm 1}) = \sigma_i^{\pm 1}$ and $\eta(\tau_i) = \sigma_i - \sigma_i^{-1}$ for $1 \leq i \leq n - 1$.
This homomorphism is one of the main ingredients of the definition of Vasiliev invariants for braids (see \cite{B1}).
L.~Paris \cite{P}
proved that the desingularization map $\eta$ is an embedding of $SM_n$ into the group algebra $\mathbb{C}[B_n]$.
It provides an answer to the question of J.~Birman \cite{B1}.
Hence, the homomorphism $\eta$ gives other solution of the word problem in $SM_n$.

To solve the word problem in the groups
$SB_n$ and hence in $PB_n$ we need to define an inverse to $\tau_i$. Take the group algebra $\mathbb{C}[[B_n]]$ of
formal power series. If an element $a \in \mathbb{C}[[B_n]]$ has the inverse $a^{-1} \in \mathbb{C}[[B_n]]$,
then
$$
(a - a^{-1})^{-1} = a (a^2 - 1)^{-1}.
$$
Since
$$
(a^2 - 1)^{-1} = -(1 + a^2 + a^4 + \ldots),
$$
we have proved the following proposition.

\begin{prop}
{\it There is a homomorphism $\widetilde{\eta} : SB_n \longrightarrow \mathbb{C}[[B_n]]$ which is defined on the generators
$$
\tau_i \mapsto \sigma_i - \sigma_i^{-1},~~~\sigma_i \mapsto \sigma_i,
$$
and is an extension of the homomorphism $\eta$. Under this homomorphism we have}
$$
\tau_i^{-1} \mapsto (\sigma_i - \sigma_i^{-1})^{-1} = - \sigma_i (1 + \sigma_i^2 + \sigma_i^4 + \ldots).
$$
\end{prop}

\begin{question}
Is it true that the homomorphism $\widetilde{\eta} : SB_n \longrightarrow \mathbb{Z}[[B_n]]$  is an embedding?
\end{question}

If it is true, then we will have a solution of the word problem for the groups
$SB_n$ and $PB_n$.

\section{Geometric Interpretation of Pseudo Braids and Alexander's Theorem}

The geometric interpretation for the generators $p_i$, $i = 1, 2, \ldots, n-1$ is Figure 1 where the singular crossing is replaced by a pre-crossing. Any
two geometric pseudo braids for the same $n$ can be composed into their product (See Figure 5). 
\begin{center}
\includegraphics[width=10.0cm]{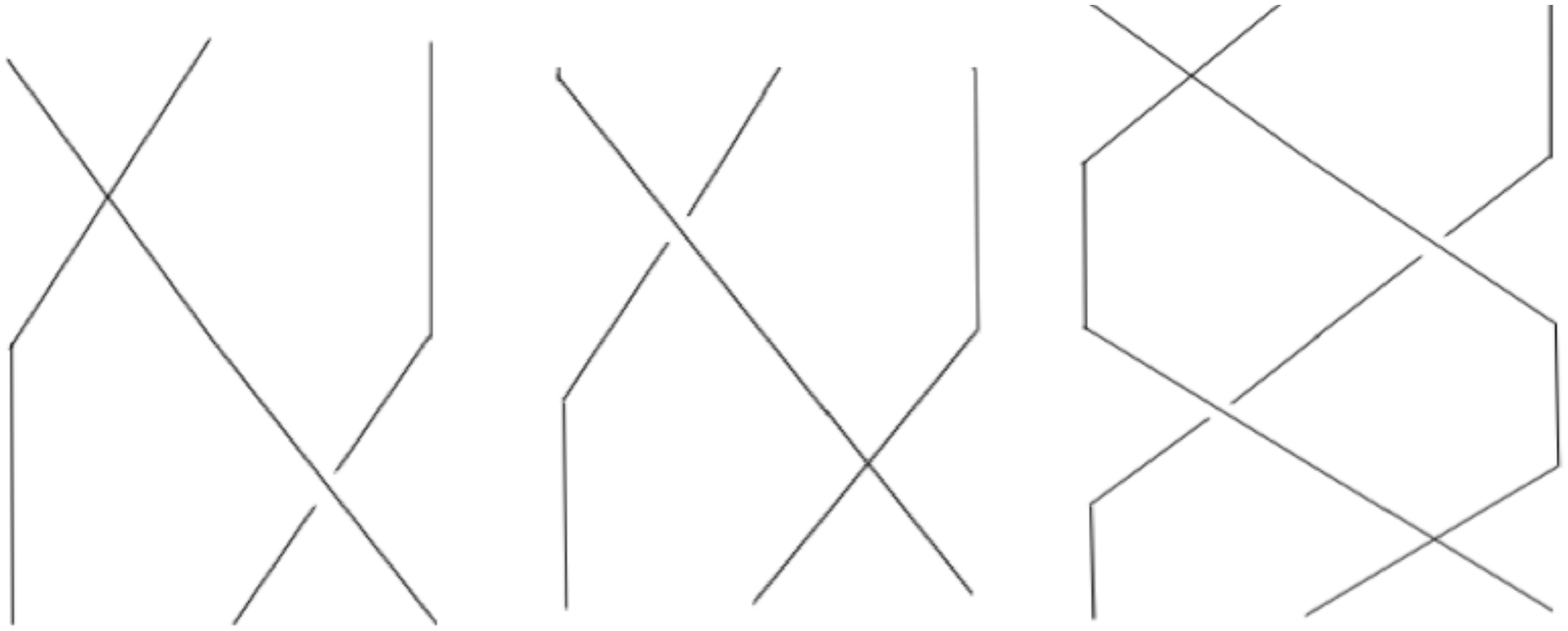}\\
{Figure 5: Product of two pseudo braids}
\end{center}
Every pseudo braid corresponds a pseudo knot or link by taking its closure (see Figure 6).
\begin{center}
\includegraphics[width=9.0cm]{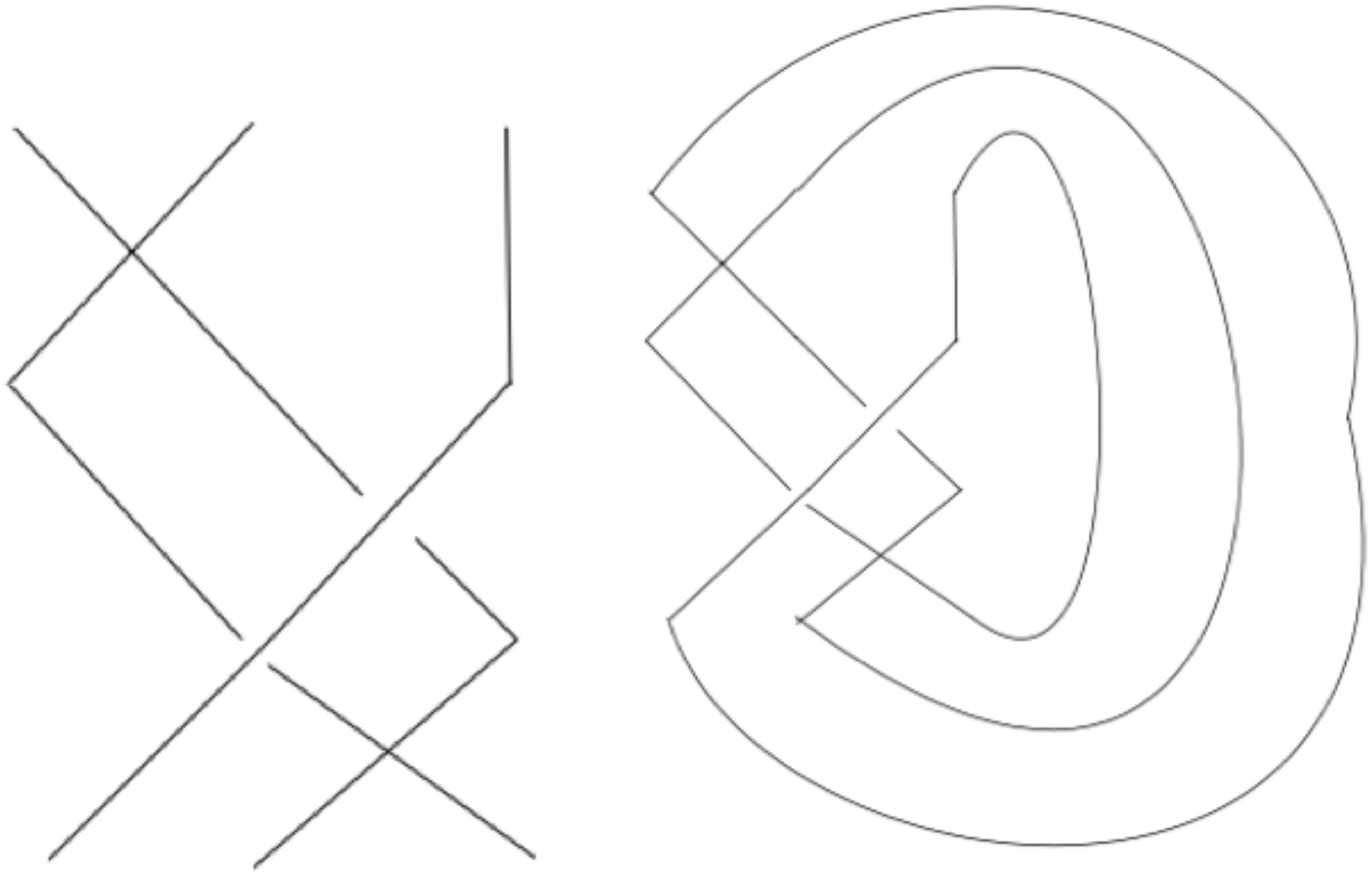}\\
{Figure 6: A pseudo braid and its closure}
\end{center}
Denote the closure of a pseudo braid $\beta$ as $\widehat{\beta}$. This closure operation gives rise to a map from the set of
pseudo braids to the set of pseudo links
$$
\widehat{~} : PM \longrightarrow P\mathcal{L},
$$
where $PM = \bigcup_{n=2}^{\infty} PM_n$, $P\mathcal{L}$ is the set of pseudo links.

We will represent pseudo links
as spatial graphs in Euclidean 3-space $\mathbb{R}^3$. This is in contrast with the definition of virtual links.
Virtual links are combinatorial objects but pseudo links are topological objects which are similar to the
singular links.

Birman \cite{B1} proved Alexander's Theorem for the singular links. This result allows us to prove the following analogue of Alexander' Theorem for pseudo links.

\begin{theorem}\label{al}
Let $L$ be a pseudo link.  Then there exists a pseudo braid $\beta \in PM_n$ for some $n$ such that the
closure $\widehat{\beta}$ is equivalent to $L$.
\end{theorem}

\begin{proof}
As there is a natural map $F$ from singular links onto pseudo links (see~(\ref{eq:F})), let $L'$ be a singular link in $F^{-1}(L).$ From~\cite{B1} we know that there is a singular braid $\beta'$ such that the closure of $\beta'$ is the singular link $L'$. From Proposition~\ref{con} there is a unique pseudo braid $\beta$ corresponding to $\beta'$ under the monoid isomorphism and the closure of $\beta$ is exactly the pseudo link $L.$
\end{proof}

\section{Markov's Theorem}

Consider the set of monoids of pseudo braids $PM_n$ for $n = 2, 3, \ldots$ and let
$PM = \bigcup_{n=2}^{\infty} PM_n$.
Define the following Markov's moves on the set $PM$:\\

M1. If $\beta \in PM_n$ and $a \in B_n$ then
$$
\beta \leftrightarrow a^{-1} \beta a.
$$\\

M2. If $\beta = \beta_1 \beta_2 \in PM_n$  then
$$
\beta \leftrightarrow \beta_2 \beta_1.
$$\\

M3. If $\beta \in PM_n$  then
$$
\beta \leftrightarrow  \beta \sigma^{\pm 1}_n \in PM_{n+1}.
$$\\

M4. If $\beta \in PM_n$  then
$$
\beta \leftrightarrow  \beta p_n \in PM_{n+1}.
$$

\begin{theorem}\label{mainth}
Let $L$ and $L'$ be two pseudo link diagrams. Suppose that $L = \widehat{\beta}$ and $L' = \widehat{\beta'}$ for
some pseudo braids $\beta \in PM_n$ and $\beta' \in PM_m$. Then the pseudo links $L$ and $L'$ are equivalent if
and only if there is a finite sequence of Markov's moves, which transform $\beta$ to $\beta'$.
\end{theorem}

To prove this theorem, we will follow ideas of Birman \cite[Chapte 2]{B} who proved the Markov theorem for classical links and ideas of
Gemein \cite{G} who proved the Markov theorem for singular links. For the singular case, the Markov's moves
consist of the moves M1-M3. Birman's proof does not use Reidemeister moves but triangular moves. Gemein defines
triangular moves which can be applied to singular points. We need to modify his moves so they can be applied to pre-crossings.

\begin{proof}
It is straight forward to check that if $\beta$ and $\beta'$ are different by one of the Markov moves M1-M4, their closures are equivalent pseudo link diagrams.

To prove the converse, it is sufficient to show that if $L'$ and $L$ differ by a generalized Reidemeister move for pseudo link diagrams, then $\beta'$ can be transformed to $\beta$ through Markov moves and pseudo braid isotopies.
Denote by $l'=f^{-1}(L')$ and $l=f^{-1}(L)$, i.e, $l', l$ are obtained from $L', L$ by replacing pre-crossings by singular crossings. By the Alexander's Theorem, there exist singular braids $\alpha', \alpha$ such that $\widehat{\alpha'}=l'$ and $\widehat{\alpha}=l$. Note that we may assume that $\alpha, \alpha'$ are obtained from $\beta, \beta'$ by replacing pre-crossings by singular crossings.  There are two cases to settle the proof:
\begin{enumerate}
\item If $L, L'$ differ by a pseudo Reidemeister move induced by a singular Reidemeister move, then $l, l'$ differ by a singular Reidemeister move.  Then by Gemein's Markov theorem for singular links \cite{G}, the singular braids $\alpha'$ and $\alpha$ are different by singular braid isotopies and the singular version of Markov moves M1-M3. Then obviously, $\beta$ and $\beta'$ differ by pseudo braid isotopies and Markov moves M1-M3.
\item If $L, L'$ differ by the first pseudo Reidemeister move (removing a pre-crossing), then we cannot use the singular moves because $l, l'$ are no longer isotopic. Note that if a link diagram placed in $\mathbb{R}^2$ differently, say differ by a rotation, then their corresponding pseudo braids are different by pseudo braid isotopies and Markov moves M1-M3 again similar to the argument of Case (1) using Gemein's theorem. So we place $L, L'$ in $\mathbb{R}^2$ in such a way that all other parts of the two diagrams are identical except for the local pieces 1 and 2 in Figure 7.
\begin{center}
\includegraphics[width=10.0cm]{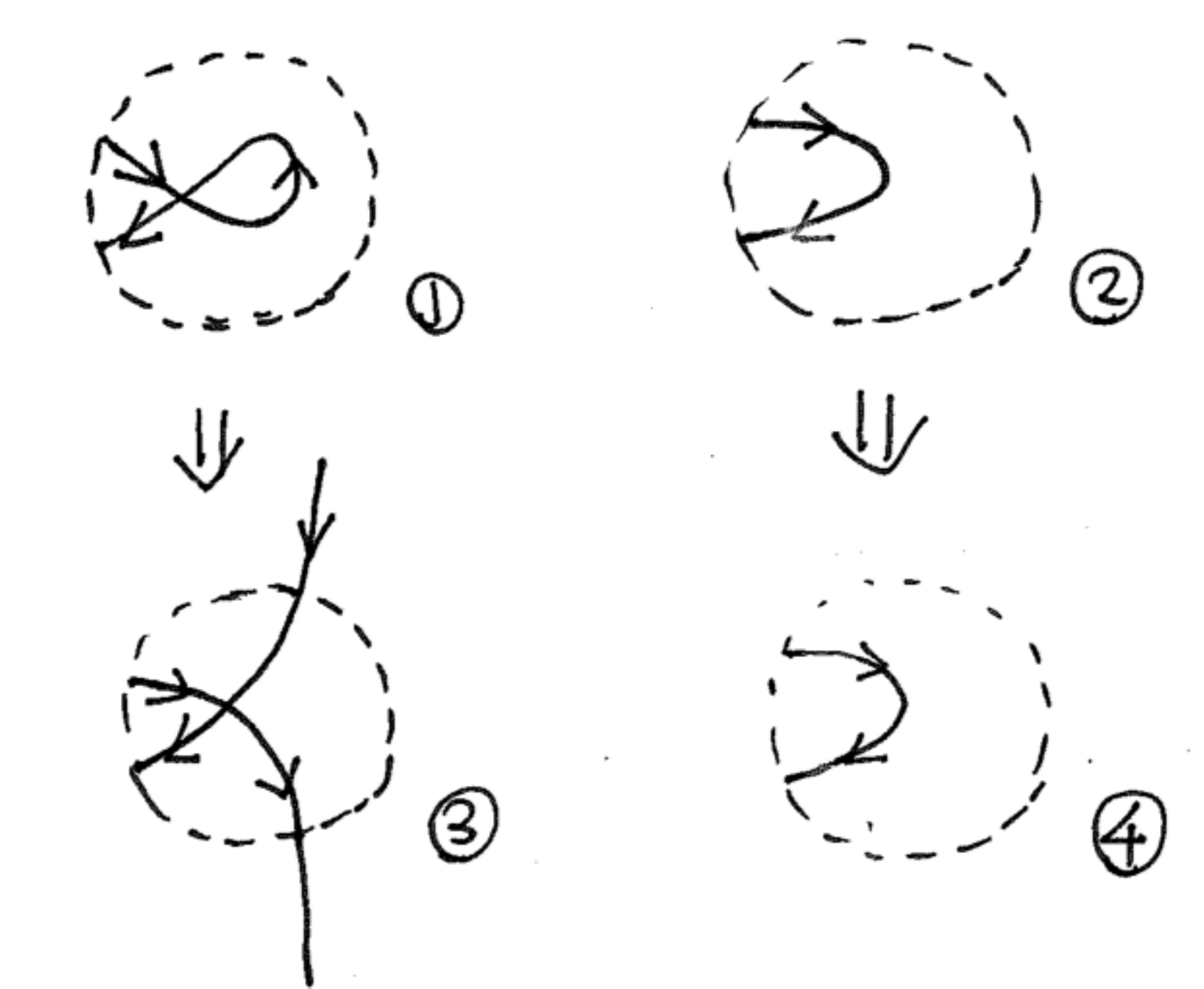}\\
{Figure 7: Braiding of the first pseudo Reidemeister move}
\end{center}
Without loss of generality, we assume $L, L'$ are made up of piece of segments so that when we orient the links by labeling for each piece of the diagram an arrow, the arrow is either up or down (no horizontal segments). After we assigns arrows to each piece of the links, we notice that there is one extra piece of up-arrow in $L'$ (see Figure 7).
Now we braid the two pseudo knot diagrams by keeping the down-arrows but stretching a pair of down arrows (a braid) to replace all the up-arrows. This new pair of braid has over-crossings to other part of the diagram. We first braid $L$ and $L'$ at the corresponding up-arrow pieces in part of the diagram where they are identical and in the end  we braid the local pieces 1 and 2 in Figure 7 into 3 and 4. In this way, the two pseudo braids $\beta$ and $\beta'$ we obtained differ exactly by the Markov move M4.
\end{enumerate}
As diagrams for the same pseudo link differ by a finite sequence of either first pseudo Reidemeister move or other pseudo Reidemeister moves induced by the singular Reidemeister moves, their corresponding braids are different by a finite sequence of Markov moves M1-M4.
\end{proof}

%
%
%
%
%
%
%
%
%
%


\begin{thebibliography}{99}
\bibitem{H}
R.~Hanaki, {\it Pseudo diagrams of links, links and spatial graphs}, Osaka J. Math., 47 (2010), 863--883.

\bibitem{J}
A. Henrich, R. Hoberg, S. Jablan, L. Johnson, E. Minten, L. Radovic, {\it The Theory of Pseudoknots},
 J. Knot Theory Ramifications, 22, no. 7 (2013),  21 pp.



\bibitem{Baez}
J. C. Baez, {\it Link invariants of finite type and perturbation
theory}, Lett. Math. Phys., 26, no. 1 (1992), 43--51.


\bibitem{B1}
J.~S.~Birman, {\it New Points of View in Knot Theory},
 Bull.  Amer. Math. Soc. (New Series), 28 (1993), 253--287.


\bibitem{FKR}
R. Fenn, E. Keyman, C. Rourke, The singular braid monoid embeds in a group,
J. Knot Theory Ramifications, 7, no. 7 (1998), 881--892.


\bibitem{Bar}
 V.~Bardakov, {\it The virtual and universal braids}, Fund. Math., 181 (2004), 1--18.

\bibitem{B}
J. S. Birman, Braids, links and mapping class group,
Princeton\noindent --Tokyo: Univ. press, 1974.

\bibitem{G}
B.~Gemein, {\it Singular braids and Markov's Theorem}, J. Knot Theory Ramifications, 6, no. 4 (1997), 441--454.

\bibitem{Cor}
R. Corran, A normal form for class of monoids including the
singular braid monoids, J. Algebra, 223, No.1 (2000), 256-282.

\bibitem{V}
V. V. Vershinin,  {\it On the singular braid monoid}. (Russian) Algebra i Analiz, 21, no. 5 (2009), 19--36; translation in St. Petersburg Math. J. 21, no. 5 (2010), 693--704.

\bibitem{P}
L. Paris, {\it The proof of Birman's conjecture on singular braid monoids}, Geom. Topol. 8 (2004), 1281--1300 (electronic).

\end{thebibliography}
\end{document}